\documentclass[11pt]{amsart}
\usepackage[T1]{fontenc}
\usepackage[english]{babel}

\usepackage{libertine}

\usepackage{amsmath,amssymb,amsthm,enumerate,xspace}
\usepackage{comment}
\usepackage[colorlinks=true,citecolor=blue,linkcolor=blue]{hyperref}
\usepackage{mathrsfs}

\numberwithin{equation}{section}

\newtheorem{main}{Theorem}[section]

\newtheorem{prop}[main]{Proposition}
\newtheorem{lem}[main]{Lemma}
\newtheorem{cor}[main]{Corollary}
\newtheorem{imain}{Theorem}

\newtheorem*{iprob*}{Problem}
\newtheorem{innercustomthm}{Theorem}
\newenvironment{customthm}[1]
  {\renewcommand\theinnercustomthm{#1}\innercustomthm}
  {\endinnercustomthm}
\theoremstyle{definition}

\newtheorem{defi}[main]{Definition}
\newtheorem{rem}[main]{Remark}
\newtheorem{rems}[main]{Remarks}
\newtheorem{exam}[main]{Example}

\newcommand{\sH}{H}

\newcommand{\R}{\mathbf{R}}
\newcommand{\N}{\mathbf{N}}

\newcommand{\bS}{\mathbf{S}}

\newcommand{\PSL}{\mathbf{PSL}}
\newcommand{\HH}{\mathbf{H}} 
\newcommand{\HHI}{\mathbf{H}^{\infty}}
\newcommand{\SO}{\mathbf{SO}} 
\newcommand{\OO}{{\mathbf O}}
\newcommand{\PO}{\mathbf{PO}}
\newcommand{\Isomi}{{\Isom}(\HHI)}
\newcommand{\Isom}{{\rm Is}}
\newcommand{\Aut}{{\rm Aut}}
\newcommand{\Mob}{\mathrm{M\ddot ob}}

\newcommand{\se}{\subseteq}
\newcommand{\inv}{^{-1}}
\newcommand{\ro}{\varrho}
\newcommand{\fhi}{\varphi}

\newcommand{\lra}{\longrightarrow}
\newcommand{\Id}{\mathrm{Id}}
\newcommand{\wt}{\widetilde}

\newcommand{\ol}{\overline}

\newcommand{\cat}[1]{{\upshape CAT({\ensuremath#1})}\xspace}

\DeclareMathOperator{\arcosh}{arcosh}
\newcommand{\bx}{B'}

\def\hyph{-\penalty0\hskip0pt\relax}

\title[Self-representations of the M\"obius group]{Self-representations of the M\"obius group}
\date{May 2018}

\author[Nicolas Monod]{Nicolas Monod}
\address{EPFL, Switzerland}
\email{nicolas.monod@epfl.ch}
\author[Pierre Py]{Pierre Py$^*$}
\address{Instituto de Matem\'aticas, Universidad Nacional Aut\'onoma de M\'exico, M\'exico}
\email{py@im.unam.mx}
\thanks{$^*$Partially supported by the french project ANR AGIRA and by project papiit IA100917 from DGAPA UNAM}

\begin{document}
\begin{abstract}
Contrary to the finite-dimensional case, the M\"obius group admits interesting self\hyph{}representations when infinite\hyph{}dimensional. We construct and classify all these self\hyph{}representations.

The proofs are obtained in the equivalent setting of isometries of Lobachevsky spaces and use kernels of hyperbolic type, in analogy to the classical concepts of kernels of positive and negative type.
\end{abstract}
\maketitle
\thispagestyle{empty}

\section{Introduction}
\subsection{Context}
For an ordinary connected Lie group, the study of its continuous \emph{self\hyph{}representations} is trivial in the following sense: every injective self\hyph{}representation is onto, and hence an automorphism. 

In the infinite\hyph{}dimensional case, another type of ``tautological'' self\hyph{}representations presents itself. Namely, the group will typically contain isomorphic copies of itself as natural proper subgroups. For instance, a Hilbert space will be isomorphic to most of its subspaces.

Remarkably, some infinite\hyph{}dimensional groups admit also completely different self\hyph{}representations which are not in any sense smaller tautological copies of themselves. This phenomenon has no analogue in finite dimensions and the simplest case is as follows.

Let $E$ be a Hilbert space and $\Isom(E)\cong E \rtimes \OO(E)$ its isometry group. To avoid the obvious constructions mentioned above, we only consider \emph{cyclic} self\hyph{}representations (in the affine sense). It is well-known that there is a whole wealth of such self\hyph{}representations. They are described by \emph{functions of conditionally negative type}. More precisely, the question becomes equivalent to describing all \emph{radial} functions of conditionally negative type on $E$ because one can arrange, by conjugating, that $\OO(E)$ maps to itself. Thus the question reduces to the study of a fascinating space of functions $\Psi\colon \R_+\to \R_+$, and new such functions can be obtained by composing a given $\Psi$ with any \emph{Bernstein function}~\cite{Schilling-Song-Vondracek_2}.

We see that this first example, $\Isom(E)$, has many --- almost too many --- self\hyph{}representations for a precise classification. How about other  infinite\hyph{}dimensional groups? Are they too rigid to admit any, or again so soft as to admit too many?

\bigskip
Considering that $\Isom(E)$ sits in the much larger \emph{M\"obius group} $\Mob(E)$ of $E$, this article answers the following questions:

\medskip
\begin{itemize}
\item  Does any non-tautological $\Isom(E)$-representation extend to $\Mob(E)$?
\item Can the irreducible self\hyph{}representations of $\Mob(E)$ be classified? 
\item Among all Bernstein functions, which ones correspond to M\"obius representations?
\end{itemize}

\medskip
In short, the answer is that the situation is much more rigid than for the isometries $\Isom(E)$, but still remains much richer than in the finite\hyph{}dimensional case. Specifically, there is exactly a one\hyph{}parameter family of self\hyph{}representations. This appears as a continuous deformation of the tautological representation, given by the Bernstein functions $x\mapsto x^t$ where the parameter $t$ ranges in the interval $(0, 1]$.

\subsection{Formal statements}
Recall that the M\"obius group $\Mob(E)$ is a group of transformations of the conformal sphere $\widehat E = E\cup\{\infty\}$; it is generated by the isometries of $E$, which fix $\infty$, and by the inversions $v \mapsto (r/\|v\|)^2 v$, where $r>0$ is the inversion radius (see e.g.~\cite[I.3]{Reshetnyak}). In particular it contains all homotheties.

A first basic formalisation of the existence part of our results is as follows. Let $E$ be an infinite\hyph{}dimensional separable real Hilbert space.

%
%

\begin{imain}\label{thm:Mobnew}
For every $0<t\leq 1$ there exists a continuous self\hyph{}representation $\Mob(E)\to \Mob(E)$
%
%
with the following properties:
\begin{itemize}
\item it restricts to an affinely irreducible self\hyph{}representation of $\Isom(E)$,
\item it maps the translation by $v\in E$ to an isometry with translation part of norm $\|v\|^t$,
\item it maps homotheties of dilation factor $\lambda$ to similarities of dilation factor $\lambda^t$. 
\end{itemize}
\end{imain}

Before turning to our classification theorem, we recast the discussion into a more suggestive geometric context by switching to the viewpoint of Lobachevsky spaces. Recall that there exists three space forms: Euclidean, spherical and real hyperbolic. In infinite dimensions, this corresponds to Hilbert spaces, Hilbert spheres and the infinite\hyph{}dimensional real hyperbolic space $\HHI$. (For definiteness, we take all our spaces to be separable in this introduction.)

The hyperbolic Riemannian metric of curvature $-1$ induces a distance $d$ on $\HHI$ and we consider the corresponding (Polish) isometry group  $\Isomi$, which also acts on  the boundary at infinity $\partial \HHI$. After choosing a point in $\partial \HHI$, we can, just like in the finite\hyph{}dimensional case, identify $\partial \HHI$ with the sphere $\widehat E$ in a way that induces an isomorphism
$$\Isomi\cong \Mob(E).$$
As will be recalled in Section~\ref{sec:prelim}, there is a  linear model for $\Isomi$, and hence the usual notion of irreducibility --- which happens to coincide with a natural geometric notion. Here is the geometric, and more precise, counterpart to Theorem~\ref{thm:Mobnew}.

\begin{customthm}{\ref{thm:Mobnew}\textsuperscript{bis}}\label{thm:exist}
For every $0<t\leq 1$ there exists a continuous irreducible self\hyph{}representation
$$\ro^\infty_t\colon \Isomi\longrightarrow \Isomi$$
and a $\ro^\infty_t$-equivariant embedding
$$f^\infty_t\colon  \HHI \longrightarrow \HHI$$
such that for all $x,y \in \HHI$ we have
\begin{equation}\label{eq:exist}
\cosh d\big(f^\infty_t(x), f^\infty_t(y)\big)  = \big(\cosh d(x, y)\big)^t.
\end{equation}
%
\end{customthm}

Intuitively, the representations $\ro^\infty_t$ are limiting objects for the representations
$$\ro^n_t\colon \Isom(\HH^n)\longrightarrow \Isomi$$
of the finite\hyph{}dimensional $\Isom(\HH^n)$ that we studied in~\cite{Monod-Py}, although the latter do not have a simple explicit expression like~\eqref{eq:exist} for the distance. Using this relation to $\ro^n_t$ and our earlier results, we can show that these $\ro^\infty_t$, which are unique up to conjugacy for each $t$, exhaust all possible irreducible self\hyph{}representations.

\begin{imain}\label{thm:uniqueness}
Every irreducible self\hyph{}representation of $\Isomi$ is conjugated to a unique representation $\ro^\infty_t$ as in Theorem~\ref{thm:exist} for some $0<t\leq 1$.
\end{imain}

Even though we stated above that $\ro^\infty_t$ is, in some sense, a limit of representations $\ro^n_t$ for $\Isom(\HH^n)$ as $n$ goes to infinity, there is a priori a difficulty in making this precise. Consider indeed a standard nested sequence of totally geodesic subspaces
$$\HH^n \se  \HH^{n+1} \se  \ldots \se \HH^m \se\ldots\se \HHI$$
with dense union. For $m\ge n$, we can restrict $\ro^m_t$ to a copy of $\Isom(\HH^n)$. After passing to the irreducible component, this creates a copy of $\ro^n_t$. For each new $m\geq n$, there is a new positive definite component that needs to be taken into account. Likewise, the associated harmonic map $f^m_t\colon \HH^m\to\HHI$ from~\cite{Monod-Py}, when restricted to $\HH^n$, does not coincide with $f^n_t$ and indeed is not harmonic on $\HH^n$; rather, it lies at finite distance from $f^n_t$.

Our strategy for eliminating all these difficulties is not to work with the groups, and not to work with the spaces either. Instead, we only keep track of the various distance functions and study their pointwise convergence. We then show that the spaces and groups can be reconstructed from this data after this easier limit has been established. This limiting behaviour of the distance functions is a basic instance of the phenomenon of concentration of measure on spheres.

This strategy is very much the same as the one behind the use of kernels of positive type to construct orthogonal representations, and of kernels of conditionally negative type to construct affine isometric actions. Specifically, we use the notion of kernels of \emph{hyperbolic type} and establish the necessary reconstruction results.


\subsection{Further considerations}
The reader will have noticed that there is no continuity assumption in Theorem~\ref{thm:uniqueness}, although the representations constructed in Theorem~\ref{thm:exist} are continuous. Indeed this formulation of Theorem~\ref{thm:uniqueness} necessitates the following result, which leverages the automatic continuity proved by Tsankov in~\cite{Tsankov13} for the orthogonal group.

\begin{imain}\label{thm:auto}
Every irreducible self\hyph{}representation of $\Isomi$ is continuous.
\end{imain}

Such an automatic continuity phenomenon can fail in finite dimensions due to the fact that algebraic groups admit discontinuous endomorphisms induced by wild field endomorphisms (see Lebesgue~\cite[p.~533]{Lebesgue07} and~\cite{Kestelman51}).

\smallskip
We do not know whether $\Isomi$ enjoys the full strength of the automatic continuity established by Tsankov for the orthogonal group.

\begin{iprob*}
Is every homomorphism from $\Isomi$ to any separable topological group continuous?
\end{iprob*}

In another direction, the existence of interesting self\hyph{}representation of $\Isomi$ raises the possibility of composing them. The geometric description of Theorem~\ref{thm:exist} suggests that the composition of $\ro^\infty_t$ with $\ro^\infty_s$ should be related to $\ro^\infty_{ts}$. Indeed, the considerations of Section~\ref{defrot} show readily that $\ro^\infty_t\circ \ro^\infty_s$ remains \emph{non-elementary} in the sense recalled below and hence admits a unique irreducible sub-representation necessarily isomorphic to $\ro^\infty_{ts}$.

In other words, Theorem~\ref{thm:uniqueness} implies that upon co-restricting to the unique irreducible part, the semi-group of irreducible self\hyph{}representations modulo conjugation is isomorphic to the multiplicative semi-group $(0,1]$. We could not ascertain that it is really necessary to extract the irreducible part. 

\begin{iprob*}
Is the composition of two irreducible self\hyph{}representations of $\Isomi$ still irreducible?
\end{iprob*}

The last section of this article adds a few elements to the study in~\cite{Monod-Py} of the representations $\ro^n_t$ of $\Isom(\HH^n)$ for $n$ \emph{finite}. 

\smallskip
\setcounter{tocdepth}{1}
\tableofcontents

\smallskip

\section{Notation and preliminaries}
\label{sec:prelim}
\subsection{Minkowski spaces}
\label{sec:Minkowski}
We work throughout over the field $\R$ of real numbers. The scalar products and norms of the various Hilbert spaces introduced below will all be denoted by ${\langle\cdot, \cdot\rangle}$ and ${\|\cdot\|}$. Given a Hilbert space $\sH$, we consider the Minkowski space $\R\oplus \sH$ endowed with the bilinear form $B$ defined by
$$B(s\oplus h, s'\oplus h') = s s' - \langle h,h'\rangle.$$
This is a ``strongly non-degenerate bilinear form of index one'' in the sense of~\cite{Burger-Iozzi-Monod}, to which we refer for more background. We endow $\R\oplus \sH$ with the topology coming from the norms of its factors. The adverb ``strongly'' refers to the fact that the corresponding uniform structure is \emph{complete}.

Given a cardinal $\alpha$, the hyperbolic space $\HH^\alpha$ can be realised as the upper hyperboloid sheet
$$\HH^\alpha=\big\{x=s\oplus h \in \R\oplus \sH : B(x,x) = 1 \text{ and } s > 0\big\},$$
where $\sH$ is a Hilbert space of Hilbert dimension $\alpha$. The visual boundary $\partial\HH^\alpha$ can be identified with the space of $B$-isotropic lines in $\R\oplus \sH$. We simply write $\HHI$ for our main case of interest, namely the separable infinite\hyph{}dimensional Lobachevsky space $\HHI=\HH^{\aleph_0}$.

The distance function associated to the hyperbolic metric is characterised by
$$\cosh d(x,y) = B(x,y)$$
and is therefore compatible with the ambient topology and complete. The group $\OO(B)$ of invertible linear operators preserving $B$ acts projectively on $\HH^\alpha$, inducing an isomorphism $\PO(B)\cong\Isom(\HH^\alpha)$. Alternatively, $\Isom(\HH^\alpha)$ is isomorphic to the subgroup of index two $\OO_+(B)<\OO(B)$ which preserves the upper sheet $\HH^\alpha$. See Proposition~3.4 in~\cite{Burger-Iozzi-Monod}.

An isometric action on $\HH^\alpha$ is called \emph{elementary} if it fixes a point in $\HH^\alpha$ or in $\partial\HH^\alpha$, or if it preserves a line. Any non-elementary action preserves a \emph{unique} minimal hyperbolic subspace and $\HH^\alpha$ is itself this minimal subspace if and only if the associated linear representation is irreducible~\cite[\S4]{Burger-Iozzi-Monod}.

\subsection{Second model}\label{sec:model}
It is often convenient to use another model for the Minkowski space $\R\oplus \sH$, as follows. Suppose given two points at infinity, represented by isotropic vectors $\xi_1, \xi_2$ such that $B(\xi_1, \xi_2)=1$. Define $E$ to be the $B$-orthogonal complement $\{\xi_1, \xi_2\}^\perp$. Then $-B$ induces a Hilbert space structure on $E$. We can now identify $(\R\oplus \sH, B)$ with the space ${\R^2\oplus E}$ endowed with the bilinear form $\bx$ defined by
$$\bx\big((s_1,s_2)\oplus v, (s_1', s_2')\oplus v'\big)= s_1 s_2' + s_2 s_1' - \langle v,v'\rangle$$
in such a way that the isomorphism takes $\xi_1, \xi_2$ to the canonical basis vectors of $\R^2$ (still denoted $\xi_i$) and that $\HH^\alpha$ is now realised as
$$\HH^\alpha=\big\{x=(s_1, s_2, v) : B'(x,x) = 1 \text{ and } s_1 > 0\big\},$$
noting that $s_1>0$ is equivalent to $s_2>0$ given the condition ${B'(x,x) = 1}$. We can further identify $\partial \HH^\alpha$ with $\widehat E=E\cup\{\infty\}$ by means of the following parametrisation by $B'$-isotropic vectors:
\begin{equation}\label{eq:para:bord}
v \longmapsto \big(\tfrac12 \|v\|^2,1\big)\oplus v,\kern10mm \infty \longmapsto \xi_1.
\end{equation}
In particular, $0\in E$ corresponds to $\xi_2$. This parametrisation intertwines the action of $\Isom(\HH^\alpha)$ with the M\"obius group of $E$. For instance, the Minkowski operator exchanging the coordinates of the $\R^2$ summand corresponds to the inversion in the sphere of radius~$\sqrt2$ around $0\in E$.

\subsection{Subspaces}\label{sec:subspace}
The hyperbolic subspaces of $\HH^\alpha$ are exactly all subsets of the form $\HH^\alpha \cap N$ where $N<\R\oplus \sH$ is a closed linear subspace of $\R\oplus\sH$. It is understood here that we accept the empty set, points and (bi-infinite) geodesic lines as hyperbolic subspaces.

\begin{defi}
The \emph{hyperbolic hull} of a subset of $\HH^\alpha$ is the intersection of all hyperbolic subspaces containing it. A subset is called \emph{hyperbolically total} if its hyperbolic hull is the whole ambient $\HH^\alpha$.
\end{defi}

Thus the hyperbolic hull of a subset $X\se \HH^\alpha$ coincides with $\HH^\alpha\cap\overline{\mathrm{span}}(X)$. It follows that $X$ is hyperbolically total if and only if it is total in the topological vector space $\R\oplus \sH$.


\medskip
There is a bijective correspondence, given by $\sH'\mapsto \HH^\alpha \cap (\R\oplus \sH')$, between Hilbert subspaces $\sH'< \sH$ and hyperbolic subspaces that contain the point $1\oplus 0$.

In the second model, $E'\mapsto \HH^\alpha \cap (\R^2\oplus E')$ is a bijective correspondence between Hilbert subspaces $E'<E$ and hyperbolic subspaces of $\HH^\alpha$ whose boundary contains both $\xi_1$ and $\xi_2$.

\subsection{Horospheres}
\label{sec:horo}
We can parametrise $\HH^\alpha$ by $\R\times E$ in the second model as follows. For $s\in \R$ and $v\in E$, define
$$\sigma_s(v) = \begin{pmatrix}
\frac 12 (e^s + e^{-s} \|v\|^2)\\
e^{-s}\\
e^{-s} v
\end{pmatrix}.$$
For any given $v$, the map $s\mapsto \sigma_s(v)$ is a geodesic line; its end for $s\to\infty$ is represented by $\xi_1$ and its other end $s\to-\infty$ by the isotropic vector $(\frac12 \|v\|^2,1)\oplus v$ of the parametrisation~\eqref{eq:para:bord} above. If on the other hand we fix $s$, then the map $\sigma_s$ is a parametrisation by $E$ of a horosphere $\sigma_s(E)$ based at $\xi_1$. Computing $B'(\sigma_s(u) ,\sigma_s(v))$, we find that the hyperbolic distance on this horosphere is given by
%
%
\begin{equation}\label{eq:horo:CNT}
\cosh d(\sigma_s(u),\sigma_s(v)) = 1 + \tfrac 12  e^{-2 s}\, \|u-v\|^2 \kern10mm(\forall\,u,v\in E).
\end{equation}
If $E'<E$ is a Hilbert subspace and $\HH'$ the corresponding hyperbolic subspace $\HH^\alpha \cap (\R^2\oplus E')$ considered in~\ref{sec:subspace}, then $\sigma_s(E')$ is $\sigma_s(E)\cap \HH'$ and coincides with a horosphere in $\HH'$ based at $\xi_1\in\partial\HH'$.

\section{Kernels of hyperbolic type}
\label{sec:kernels}
\subsection{The notion of kernel of hyperbolic type}
Kernels of positive and of conditionally negative type are classical tools for the study of embeddings into spherical and Euclidean spaces respectively (see e.g.\ Appendix~C in~\cite{Bekka-Harpe-Valette}). The fact that a similar notion is available for hyperbolic spaces seems to be well-known, see e.g~\S5 in~\cite{Gromov_conc}. We formalise it as follows.

\begin{defi}\label{def:KHT}
Given a set $X$, a function $\beta\colon X\times X\to\R$ is a \emph{kernel of hyperbolic type} if it is symmetric, non-negative, takes the constant value~$1$ on the diagonal, and satisfies
\begin{equation}\label{eq:KHT}
\sum_{i,j=1}^n c_i c_j \beta(x_i, x_j) \leq \Big(\sum_{k=1}^n c_k \beta(x_k, x_0)\Big)^2
\end{equation}
for all $n\in \N$, all $x_0, x_1, \ldots, x_n\in X$ and all $c_1, \ldots, c_n\in \R$.
\end{defi}

\begin{rems}\label{rems:KHT}
(i)~The case $n=1$ of~\eqref{eq:KHT} implies $\beta(x,y)\geq 1$ for all $x,y$ and $\beta\equiv 1$ is a trivial example.

(ii)~The set of kernels of hyperbolic type on $X$ is closed under pointwise limits.

(iii)~One can check that for every kernel $\psi$ of \emph{conditionally negative type}, the kernel $\beta=1+\psi$ is of hyperbolic type. However, the geometric interpretation established below
shows that this only provides examples that are in a sense degenerate; see Proposition~\ref{prop:horo:neg}.
\end{rems}

Just as in the spherical and Euclidean cases, the above definition is designed to encapsulate the following criterion.

\begin{prop}\label{prop:KHT}
Given a map $f$ from a non-empty set $X$ to a hyperbolic space, the function $\beta$ defined on $X\times X$ by $\beta(x,y)=\cosh d(f(x),f(y))$ is a kernel of hyperbolic type.

Conversely, any kernel of hyperbolic type arises from such a map $f$ that has a hyperbolically total image.
\end{prop}

A straightforward way to establish Proposition~\ref{prop:KHT} is to use the relation with kernels of positive type, as exposed in Section~\ref{sec:KHT-KPT}. However, this is unsatisfactory for one very important aspect, namely the question of naturality. In particular, how do transformations of $X$ correspond to isometries of the hyperbolic space?

Indeed an important difference between the above definition of kernels of hyperbolic type and the classical spherical and Euclidean cases is that~\eqref{eq:KHT} is asymmetric. An additional argument allows us to answer this question  --- as follows.

\begin{main}\label{thm:KHT-char}
Let $X$ be a non-empty set with a kernel of hyperbolic type $\beta$. Then the space $\HH^\alpha$ and the map $f\colon X\to \HH^\alpha$ granted by Proposition~\ref{prop:KHT} are unique up to a unique isometry of hyperbolic spaces.

Therefore, denoting by $\Aut(X,\beta)$ the group of bijections of $X$ that preserve $\beta$, there is a canonical representation $\Aut(X,\beta) \to \Isom(\HH^\alpha)$ for which $f$ is equivariant.
\end{main}

A function $F\colon G \to \R$ on a group $G$ will be called a \emph{function of hyperbolic type} if the kernel $(g,h)\mapsto F(g\inv h)$ is of hyperbolic type. In other words, this is equivalent to the data of a left-invariant kernel of hyperbolic type on $G$. Therefore, the above results imply readily the following.

\begin{cor}\label{cor:KHT-group}
For every function of hyperbolic type $F\colon G \to \R$ there is an isometric $G$-action on a hyperbolic space and a point $p$ of that space such that
$$ F(g) = \cosh d(gp, p)$$
holds for all $g\in G$ and such that the orbit $G p$ is hyperbolically total.

If moreover $G$ is endowed with a group topology for which $F$ is continuous, then the $G$-action is continuous.
\end{cor}

An interesting example is provided by the Picard--Manin space associated to the Cremona group:

\begin{exam}\label{ex:PM}
Let ${\rm Bir}(\mathbf{P}^{2})$ be the Cremona group and $\deg \colon {\rm Bir}(\mathbf{P}^{2}) \to \mathbf{N}$ be the degree function. That is, $\deg(g)\geq 1$ is the degree of the homogeneous polynomials defining $g\in {\rm Bir}(\mathbf{P}^{2})$. It follows from the work of Cantat~\cite{Cantat} that the function $\deg$ is of hyperbolic type, the associated hyperbolic space being the Picard--Manin space (see also Chap.~5 in~\cite{Manin}).
\end{exam}

\subsection{The geometric characterization of kernels of hyperbolic type}\label{sec:KHT-KPT}
In order to prove Proposition~\ref{prop:KHT}, we recall that a function $N\colon X\times X\to\R$ is called a (real) \emph{kernel of positive type} if it is symmetric and satisfies $\sum_{i,j} c_i c_j N(x_i, x_j)\geq 0$ for all $n\in \N$, all $x_1, \ldots, x_n\in X$ and all $c_1, \ldots, c_n\in \R$.

If $h\colon X\to \sH$ is any map to a (real) Hilbert space $\sH$, then $N(x,y)=\langle h(x), h(y)\rangle$ defines a kernel of positive type. Conversely, the \emph{GNS construction}
associates canonically to any kernel of positive type $N$ a Hilbert space $\sH$ and a map $h\colon X\to \sH$ such that $N(x,y)=\langle h(x), h(y)\rangle$ holds for all $x,y\in X$ and such that moreover $h(X)$ is total in $\sH$ (we refer again to Appendix~C in~\cite{Bekka-Harpe-Valette}).

Now the strategy is simply to identify the hyperboloid $\HH^\alpha$ in a Minkowski space $\R\oplus \sH$ with the Hilbert space factor $\sH$. This is the \emph{Gans model}~\cite{Gans66}; the drawback is that any naturality is lost.

\begin{proof}[Proof of Proposition~\ref{prop:KHT}]
We first verify that, given a map $f$ from $X$ to a hyperbolic space $\HH^\alpha$, the kernel defined by $\beta(x,y)=\cosh d(f(x),f(y))$ is indeed of hyperbolic type. In the ambient Minkowski space $\R\oplus \sH$ for $\HH^\alpha$, the reverse Schwarz inequality implies
$$B(v, v) \leq B(v, v_0)^2 \kern10mm \forall\, v, v_0\in  \R\oplus \sH \text{ with } B(v_0, v_0)=1.$$
This can also be verified directly by using the transitivity properties of $\OO(B)$ to reduce it to the case $v_0=1\oplus 0$, where it is trivial. Given now $x_0, x_1, \ldots, x_n\in X$ and $c_1, \ldots, c_n\in \R$, we apply this inequality to $v_0 = f(x_0)$ and $v=\sum_{k=1}^n c_k f(x_k)$ and the inequality~\eqref{eq:KHT} follows.

\medskip
We turn to the converse statement; let thus $\beta$ be an arbitrary kernel of hyperbolic type on $X$. Pick $x_0\in X$ and consider the kernel $N$ on $X$ defined by
\begin{equation}\label{eq:N}
N(x,y)  = \beta(x, x_0) \beta(y, x_0) - \beta(x, y).
\end{equation}
Condition~\eqref{eq:KHT} is precisely that $N$ is of positive type. Consider thus $h\colon X\to \sH$ as given by the GNS construction for $N$ and the corresponding hyperbolic space $\HH^\alpha$ in $\R\oplus\sH$. Define $f\colon X\to \HH^\alpha$ by
$$f(x) = \beta(x_0, x) \oplus h(x).$$
Using~\eqref{eq:N}, we obtain the desired relation $B(f(x),f(y))=\beta(x,y)$.

Finally, we prove that $f(X)$ is hyperbolically total; let thus $V\se \R\oplus\sH$ be the closed linear subspace spanned by $f(X)$ and recall that it suffices to show that $V$ is all of $\R\oplus\sH$. The definition of $N$ implies $N(x_0, x_0)=0$. Therefore, we have $h(x_0)=0$ and hence $V$ contains $1\oplus 0$. Thus $V=\R\oplus\sH$ follows from the fact that $h(X)$ is total in $\sH$.
\end{proof}

\begin{rem}\label{rem:GNS}
The construction of $f$ shows that when a topology is given on $X$, the map $f$ will be continuous as soon as the kernel $\beta$ is so. Indeed, the corresponding statement holds for kernels of positive type, see e.g.\ Theorem~C.1.4 in~\cite{Bekka-Harpe-Valette}.
\end{rem}

\subsection{Functoriality and kernels}
We now undertake the proof of Theorem~\ref{thm:KHT-char}. We keep the notations introduced in the proof of Proposition~\ref{prop:KHT} for the construction of the space $\HH^\alpha\se \R\oplus\sH$ and of the map $f\colon X\to \HH^\alpha$. In order to prove Theorem~\ref{thm:KHT-char}, it suffices to give another construction of $\R\oplus\sH$, of $B$ and of $f$ that depends functorially on $(X, \beta)$, and only on this.

\begin{rem}
The previous construction introduced a choice of $x_0$, and hence of $N$ in~\eqref{eq:N}, to define $f$. We now argue more functorially, but the price to pay is that the nature of the constructed bilinear form is unknown until we compare it to the non-functorial construction.
\end{rem}

We record the following fact, wherein $\overline{V}$ denotes the completion with respect to the uniform structure induced by the given non-degenerate quadratic form of finite index. The statement follows from the discussion in~\S2 of~\cite{Burger-Iozzi-Monod}, although it is not explicitly stated in this form.

\begin{prop}\label{completion}
Let $(V,Q)$ be a real vector space endowed with a non-degenerate quadratic form of finite index. Then there is a vector space $\overline{V}$ with a strongly non-degenerate quadratic form $\overline{Q}$ of finite index equal to that of $Q$ such that $V$ embeds densely in $\overline{V}$ with $\overline{Q}$ extending $Q$.  The quadratic space $(\overline{V},\overline{Q})$ is unique up to isometry; any isometry of $(V,Q)$ extends to an isometry of $(\overline{V},\overline{Q})$.\qed
\end{prop} 

We now start our functorial construction. We extend $\beta$ to a symmetric bilinear form on the free vector space $\R[X]$ on $X$. We denote by $W_0$ the quotient of $\R[X]$ by the radical $R$ of this bilinear form and by $\wt B$ the symmetric bilinear form thus induced on $W_0$. We further denote by $\wt f$ the composition of the canonical maps $X\to \R[X]\to W_0$. In particular, we have
$$\beta(x,y) = \wt B\big(\wt f(x), \wt f(y)\big)$$
for all $x,y\in X$.

To finish the proof, it suffices to establish the following two claims. First, $W_0$ is non-degenerate of index $1$, and hence has a completion $W$ by Proposition~\ref{completion}. Secondly, there is an identification of $W$ with $\R\oplus\sH$ that intertwines $\wt f$ with $f$ and $\wt B$ with $B$, though we may of course now use $f$ to construct this identification.

In fact, we shall prove both claims at once by exhibiting an injective linear map $\iota\colon W_0 \to \R\oplus\sH$ such that
$$B\big(\iota(u), \iota(v)\big) = \wt B(u,v) \kern3mm \text{and}\kern3mm \iota(\wt f(x)) = f(x)$$
holds for all $u,v\in W_0$ and all $x\in X$; the latter property implies in particular that $\iota(W_0)$ is dense in $\R\oplus\sH$ since $f(X)$ is hyperbolically total, see Section~\ref{sec:subspace}. Although $\iota$ will be constructed using $f$, the first claim still holds because this construction implies in particular that $\wt B$ is a non-degenerate form of index one and that the completion $W$ coincides with $\R\oplus\sH$.

We turn to the construction of $\iota$. Extend $f$ by linearity to a map
$$\ol f\colon\R[X] \longrightarrow \R\oplus H.$$
Denote by $K$ the kernel of $\ol f$; by construction, $K$ is contained in the radical $R$. We now check that in fact $K=R$; let thus $\lambda = \sum_{x\in X} \lambda_x [x]$ be a finite formal linear combination of elements of $X$ and assume $\lambda\in R$. If $\ol f(\lambda)$ did not vanish, then $\ol f(\lambda)^\perp$ would be a proper subspace of $\R\oplus\sH$. However, this subspace always contains $f(X)$ since $\lambda\in R$, and thus we would contradict the fact that $f(X)$ is total in $\R\oplus\sH$. At this point, $\ol f$ induces a map $\iota$ with all the properties that we required.

This completes the proof of Theorem~\ref{thm:KHT-char}.\qed

\begin{proof}[Proof of Corollary~\ref{cor:KHT-group}]
We apply Theorem~\ref{thm:KHT-char} to the kernel $\beta$ defined on $G$ by $\beta(g,h)= F(g\inv h)$. Viewing $G$ as a subgroup of $\Aut(G,\beta)$, we obtain a homomorphism $\ro\colon G\to \Isom(\HH^\alpha)$ and a $\ro$-equivariant map $f\colon G\to \HH^\alpha$. This means that $f$ is the orbital map associated to the point $p=f(e)$. It only remains to justify the continuity claim. Since $f(G)=G p$ is hyperbolically total, the orbital continuity follows readily from the continuity of $f$, noted in Remark~\ref{rem:GNS}, because isometric actions are uniformly equicontinuous. The latter fact also implies that orbital continuity is equivalent to joint continuity for isometric actions.
\end{proof}

\subsection{Powers of kernels of hyperbolic type}
The fundamental building block for exotic self\hyph{}representations is provided by the following statement.

\begin{main}\label{thm:power}
If $\beta$ is a  kernel of hyperbolic type, then so is $\beta^t$ for all $0\leq t\leq 1$.
\end{main}

After we established this result, another proof was found, purely computational; it will be presented in~\cite{MonodFHT}.

\begin{proof}[Proof of Theorem~\ref{thm:power}]
We can assume $t>0$ since the constant function~$1$ satisfies Definition~\ref{def:KHT} trivially. In view of Proposition~\ref{prop:KHT}, it suffices to prove that for any hyperbolic space $\HH^\alpha$ with distance $d$, where $\alpha$ is an arbitrary cardinal, the kernel
$$(\cosh d)^t\colon \HH^\alpha\times \HH^\alpha \lra \R$$
is of hyperbolic type. Definition~\ref{def:KHT} considers finitely many points at a time, which are therefore contained in a finite\hyph{}dimensional hyperbolic subspace of $\HH^\alpha$ (see e.g.\ Remark~3.1 in~\cite{Burger-Iozzi-Monod}). For this reason, it suffices to prove the above statement for $\HH^m$ with $m\in \N$ arbitrarily large --- but fixed for the rest of this proof.

Given an integer $n\geq m$, we choose an isometric embedding $\HH^m\se \HH^n$ and consider the map
$$f^n_t\colon \HH^n \lra \HHI$$
that we provided in Theorem~C of~\cite{Monod-Py} (it was simply denoted by $f_t$ in that reference, but now we shall soon let $n$ vary). Consider the kernel
$$\beta_n\colon \HH^m \times \HH^m \lra \R, \kern10mm \beta_n(x,y) = \cosh d\big(f^n_t(x), f^n_t(y)\big)$$
obtained by restriction to $\HH^m\se \HH^n$; it is of hyperbolic type by Proposition~\ref{prop:KHT}. The proof will therefore be complete if we show that $\beta_n$ converges pointwise to $(\cosh d)^t$ on $\HH^m \times \HH^m$.

Choose thus $x,y\in \HH^m$. We computed an integral expression for the quantity $\beta_n(x,y)=\cosh d\big(f^n_t(x), f^n_t(y)\big)$ in \S3.B and \S3.C of~\cite{Monod-Py}. Namely, writing $u=d(x,y)$ we established
$$\beta_n(x,y) =\int_{\bS^{n-1}}\big(\cosh(u)-b_{1}\sinh(u)\big)^{-(n-1+t)} \, db$$
where $db$ denotes the integral against the normalized volume on the sphere $\bS^{n-1}$ and $b_1$ is the first coordinate of $b$ when $b$ is viewed as a unit vector in $\R^n$. We further recall (see~\cite[(3.vi)]{Monod-Py}) that
$$\big(\cosh(u)-b_{1}\sinh(u) \big)^{-(n-1)}$$
is the Jacobian of some hyperbolic transformation $g_u\inv$ of $\bS^{n-1}$. We can therefore apply the change of variable formula for $g_u$ and obtain
\begin{equation}\label{eq:I_u}
\begin{split}
\beta_n(x,y) &=\int_{\bS^{n-1}} \fhi(g_u b) \, db, \kern3mm \text{where}\\
\fhi(b) &=\big(\cosh(u)-b_{1}\sinh(u) \big)^{-t}.
\end{split}
\end{equation}
The transformation $g_u$ is given explicitly in~\cite{Monod-Py}, namely it is $g_{u}=g_{e^{u},0,\Id}$ as defined in \S2.A of~\cite{Monod-Py}. These formulas show that the first coordinate of $g_u b$ is
$$(g_u b)_1 = \frac{\sinh(u)+ b_1 \cosh(u)}{\cosh(u) + b_1 \sinh(u) }.$$
Entering this into~\eqref{eq:I_u}, we readily compute
$$\beta_n(x,y) =\int_{\bS^{n-1}} \big(\cosh(u)+b_{1}\sinh(u) \big)^{t}\, db.$$
We are thus integrating on $\bS^{n-1}$ a continuous function depending only upon the first variable $b_1$ and which is now independent of $n$. Therefore, when $n$ tends to infinity, the concentration of measure principle implies that this integral converges to the value of that function on the equator $\{b_1=0\}$. Since this equatorial value is $(\cosh(u))^t$, we have indeed proved that $\beta_n(x,y)$ converges to $(\cosh(d(x,y))^t$, as was to be shown.
\end{proof}

\section{On representations arising from kernels}
\subsection{General properties}
Let $G$ be a group and $F\colon G\to \R$ a function of hyperbolic type. According to Corollary~\ref{cor:KHT-group}, this gives rise to an isometric $G$-action on a hyperbolic space $\HH^\alpha$ together with a point $p\in \HH^\alpha$ whose orbit is hyperbolically total in $\HH^\alpha$ and such that
$$ F(g) = \cosh d(gp, p) \kern10mm (\forall g\in G).$$
We now investigate the relation between the geometric properties of this $G$-action and the properties of the function $F$.

The Cartan fixed-point theorem, in the generality presented e.g.\ in~\cite[II.2.8]{Bridson-Haefliger}, implies the following.

\begin{lem}\label{lem:Cartan}
The function $F$ is bounded if and only if $G$ fixes a point in $\HH^\alpha$.\qed
\end{lem}

Fixed points at infinity are a more subtle form of elementarity for the $G$-action; we begin with the following characterization for kernels.

\begin{prop}\label{prop:horo:neg}
Let $\beta$ be an unbounded kernel of hyperbolic type on a set $X$ and consider the map $f\colon X\to\HH^\alpha$ granted by Proposition~\ref{prop:KHT}.

Then $f(X)$ is contained in a horosphere if and only if $\beta-1$ is of conditionally negative type. 
\end{prop}

In particular we deduce the corresponding characterization for the $G$-actions.

\begin{cor}\label{cor:horo:neg}
Suppose $F$ unbounded. Then the orbit $G p$ is contained in a horosphere if and only if $F-1$ is of conditionally negative type.\qed
\end{cor}

\begin{proof}[Proof of Proposition~\ref{prop:horo:neg}]
Suppose that $f(X)$ is contained in a horosphere. We can choose the model described in Section~\ref{sec:horo} in such a way that this horosphere is $\sigma_0(E)$ in the notations of that section. Therefore, the formula~\eqref{eq:horo:CNT} implies for all $x,y\in X$ the relation
$$\beta(x, y) = \cosh d(f(x), f(y))  = 1 +  \tfrac 12 \, \big\| \sigma_0\inv (f(x)) - \sigma_0\inv (f(y))\big \|^2,$$
where $\|\cdot\|$ is the norm of the Hilbert space $E$ parametrising the horosphere. This witnesses that $\beta-1$ is of conditionally negative type.

Conversely, if $\beta-1$ is of conditionally negative type, then the usual affine GNS construction (see e.g.\ {\S}C.2 in~\cite{Bekka-Harpe-Valette}) provides a Hilbert space $E'$ and a map $\eta\colon X\to E'$ such that $\eta(X)$ is total in $E'$ and such that
$$\beta(x,y) -1 = \tfrac 12 \, \| \eta(x) - \eta(y) \|^2$$
holds for all $x,y$. Now $\sigma_0\circ \eta$ is a map to a horosphere in the hyperbolic space $\HH'$ corresponding to $E'$ in the second model (Section~\ref{sec:model}), centered at $\xi_1\in\partial \HH'$. Let $\HH''\se\HH'$ be the hyperbolic hull of $\sigma_0\circ \eta(X)$ and observe that its boundary contains $\xi_1$ since $\beta$ is unbounded. Thus $\sigma_0\circ \eta(X)$ is contained in a horosphere of $\HH''$. By Theorem~\ref{thm:KHT-char}, $\sigma_0\circ \eta$ can be identified with $f$ and hence the conclusion follows.
\end{proof}

\subsection{Individual isometries}
The \emph{type} of an individual group element for the action defined by $F$ can be read from $F$. Recall first that the \emph{translation length} $\ell(g)$ associated to any isometry $g$ of any metric space $Y$ is defined by
$$\ell(g) = \inf \big\{d(g y, y) : y\in Y\big\}.$$
We now have the following trichotomy.

\begin{prop}\label{prop:tricho}
For any $g\in G$, the action defined by $F$ satisfies
$$\ell(g) = \ln \Big(\lim_{n\to\infty} F(g^n)^{\frac1n} \Big).$$
Moreover, exactly one of the following holds.
\begin{enumerate}
\item $F(g^n)$ is uniformly bounded over $n\in\N$; then $g$ is \emph{elliptic}: it fixes a point in $\HH^\alpha$.\label{pt:elliptic}
\item $F(g^n)$ is unbounded  and $\ell(g)=0$; then $g$ is \emph{neutral parabolic}: it fixes a unique point in $\partial \HH^\alpha$ and preserves all corresponding horospheres but has no fixed point in $\HH^\alpha$.\label{pt:parabolic}
\item $\ell(g)>0$; then $g$ is \emph{hyperbolic}: it preserves a unique geodesic line in $\partial \HH^\alpha$ and translates it by $\ell(g)$.\label{pt:hyperbolic}
\end{enumerate}
\end{prop}

\begin{rem}
Consider the Picard--Manin space associated to the Cremona group ${\rm Bir}(\mathbf{P}^{2})$ as mentioned in Example~\ref{ex:PM}. Recall that the limit $\lim_{n\to\infty} \deg(g^n)^{1/n}$ is the \emph{dynamical degree} of the birational transformation $g$. Thus we see that the translation length is the logarithm of the dynamical degree, which is a basic fact in the study of the Picard--Manin space.
\end{rem}

\begin{proof}[Proof of Proposition~\ref{prop:tricho}]
Since $\arcosh (F(g^n)) = d(g^n p, p)$, we see that
$$\ln \Big(\lim_{n\to\infty} F(g^n)^{\frac1n} \Big) = \lim_{n\to\infty} \tfrac1n d(g^n p, p).$$
Now the statements of the proposition hold much more generally. Recall that if $p$ is a point of an arbitrary \cat0 space $Y$ on which $G$ acts by isometries, then the translation length of $g\in G$ satisfies
\begin{equation}\label{eq:length}
\ell(g) = \lim_{n\to\infty} \tfrac1n d(g^n p, p),
\end{equation}
see e.g.~Lemma~6.6(2) in~\cite{Ballmann-Gromov-Schroeder}. If in addition $X$ is complete and \cat{-1}, then the above trichotomy holds, see for instance \S4 in~\cite{Burger-Iozzi-Monod}. 
\end{proof}

\section{Self-representations of \texorpdfstring{$\Isomi$}{Is(H)}}
\subsection{Definition of \texorpdfstring{$\ro^\infty_t$}{}}\label{defrot}
We choose a point $p_1\in\HHI$ and consider the corresponding function of hyperbolic type $F_1$ given by the tautological representation of $\Isom(\HHI)$ on $\HHI$. We denote by $\OO$ the stabiliser of $p_1$, which is isomorphic to the infinite\hyph{}dimensional orthogonal group.

Fix $0<t\leq 1$. By Theorem~\ref{thm:power}, the function $F_t=(F_1)^t$ is still of hyperbolic type. Appealing to Corollary~\ref{cor:KHT-group}, we denote by $\ro^\infty_t$ the corresponding representation on $\HH^\alpha$.

Observe that $\alpha\leq\aleph_0$ since $F_t$ is continuous. It follows that $\alpha=\aleph_0$ since $\Isom(\HHI)$ has no non-trivial finite\hyph{}dimensional representation (this is already true for $\OO$ since $\SO(n)$ has no non-trivial representation of dimension~$<n$ for large $n$). Hence we write $\HHI$ for $\HH^\alpha$. Given $g\in\Isomi$, we write $\ell_t(g)$ for its translation length as an isometry under the representation $\ro^\infty_t$ so as not to confuse it with its translation length under the tautological representation --- which we can accordingly denote by $\ell_1(g)$.

\smallskip
Now Proposition~\ref{prop:tricho} has the following consequence.

\begin{cor}\label{cor:type}
We have $\ell_t = t\, \ell_1$ and the representation $\ro^\infty_t$ preserves the type of each element of $\Isom(\HHI)$.\qed
\end{cor}

We can further deduce the following.

\begin{cor}\label{cor:non-elem}
The representation $\ro^\infty_t$ is non-elementary.
\end{cor}

We shall also prove that $\ro^\infty_t$ is irreducible, but it will be more convenient to deduce it later on.

\begin{proof}[Proof of Corollary~\ref{cor:non-elem}]
Suppose for a contradiction that $\ro^\infty_t$ is elementary. It cannot fix a point since $F_t$ is unbounded. Thus it either fixes a point at infinity or preserves a geodesic line. Choose a copy of $\PSL_2(\R)$ in $\Isomi$; being perfect, it fixes a point at infinity in either cases and preserves the corresponding horospheres. This is however impossible because a hyperbolic element of $\PSL_2(\R)$ must remain hyperbolic under $\ro^\infty_t$ by Corollary~\ref{cor:type}.
\end{proof}

\subsection{Restricting to finite dimensions}
We begin with a general property of $\Isomi$.

\begin{prop}\label{prop:BO}
Let $L\cong \R$ be a one\hyph{}parameter subgroup of hyperbolic elements of $\Isomi$. Then an arbitrary isometric $\Isomi$-action on a metric space has bounded orbits if and only if $L$ has bounded orbits.
\end{prop}

\begin{proof}
Let $p$ be a point on the axis associated to $L$ and let $\OO$ be the stabiliser of $p$, which is isomorphic to the infinite\hyph{}dimensional (separable) orthogonal group. Then we have the Cartan-like decomposition $\Isomi= \OO L \OO$; indeed, this follows from the transitivity of $\OO$ on the space of directions at the given point $p$. On the other hand, any isometric action of  $\OO$ on any metric space has bounded orbits, see~\cite[p.~190]{Ricard-Rosendal}. The statement follows.
\end{proof}

In order to restrict a representation to finite\hyph{}dimensional subgroups, we choose an exhaustion of our Minkowski space  $\R\oplus \sH$ by finite\hyph{}dimensional Minkowski subspaces, for instance by choosing a nested sequence of subspaces $\R^n\se H$ with dense union. This gives us a nested sequence of totally geodesic subspaces
$$\HH^n \se  \HH^{n+1} \se  \ldots \se \HHI$$
with dense union, together with embeddings $\Isom(\HH^n) < \Isomi$ preserving $\HH^n$. Moreover, the union of the resulting nested sequence of subgroups
$$\Isom(\HH^n) \se  \Isom(\HH^{n+1}) \se  \ldots \se \Isomi$$
is dense in $\Isomi$: this can e.g.\ be deduced from the density of the union of all $\OO(n)$ in $\OO$ together with a Cartan decomposition (as introduced above) for some $L\se \Isom(\HH^2)$.

\begin{prop}\label{prop:restriction}
Any continuous non-elementary self\hyph{}representation of $\Isomi$ remains non-elementary when restricted to $\Isom(\HH^n)$ for any $n\geq 2$.
\end{prop}

\begin{proof}
It suffices to show that the representation remains non-elementary when restricted to the connected component $\Isom(\HH^2)^\circ\cong \PSL_2(\R)$ of $\Isom(\HH^2)$; suppose otherwise. We denote by $g$ a non-trivial element of the one\hyph{}parameter subgroup $L$ of hyperbolic elements represented by $t\mapsto \left[\begin{smallmatrix}e^t&0\\ 0&e^{-t}\end{smallmatrix}\right]$ and by $w$ the involution represented by $\left[\begin{smallmatrix}0&-1\\ 1&0\end{smallmatrix}\right]$. By Proposition~\ref{prop:BO} and continuity, $g$ cannot fix a point in $\HHI$. Therefore our apagogical assumption implies that $\Isom(\HH^2)^\circ$ either fixes a point in $\partial\HHI$ or preserves a geodesic line. Since $\Isom(\HH^2)^\circ$ is perfect, the former case holds anyway; let thus $\xi\in\partial\HHI$ be a point fixed by $\Isom(\HH^2)^\circ$. Now $g$ cannot act hyperbolically because otherwise it would have exactly two fixed points at infinity exchanged by $w$ because $w$ conjugates $g$ to $g\inv
 $; this would contradict the fact that both $g$ and $w$ fix $\xi$. Thus $g$ is parabolic.

We claim that $\xi$ is in fact fixed by $\Isom(\HH^n)^\circ$ for all $n\geq 2$. First, we know that $\Isom(\HH^n)$ acts elementarily, because otherwise Proposition~2.1 from~\cite{Monod-Py} would imply that $g$ acts hyperbolically. Next, $\Isom(\HH^n)$ cannot fix a point in $\HHI$ since $g$ does not. Thus $\Isom(\HH^n)$ fixes a point at infinity or preserves a geodesic and we conclude again by perfectness of $\Isom(\HH^n)^\circ$ that $\Isom(\HH^n)^\circ$ fixes some point in $\partial\HHI$. This point has to be $\xi$ because $g$, being parabolic, has a \emph{unique}
 fixed point at infinity. This proves the claim.

Finally, no other point than $\xi$ is fixed by $\Isom(\HH^n)^\circ$ since this holds already for $g$. Therefore, $\xi$ is also fixed by $\Isom(\HH^n)$ since the latter normalises $\Isom(\HH^n)^\circ$. Therefore $\Isomi$ fixes $\xi$ by density. This contradicts the assumption that the self\hyph{}representation was non-elementary. 
\end{proof}

\subsection{Completion of the proofs}
A crucial remaining step is the following result, which relies notably on our classification from~\cite{Monod-Py}.

\begin{main}\label{thm:exunic}
Choose a point $p_0\in \HHI$. For every irreducible continuous self\hyph{}representation $\ro$ of $\Isomi$ there is $0<t\leq 1$ and $p\in \HHI$ such that
\begin{equation}\label{eq:exunic}
\cosh d\big(\ro(g) p, \ro(h) p\big)  = \big(\cosh d(g p_0, h p_0)\big)^t  
\end{equation}
holds for all $g,h\in \Isomi$.
\end{main}

Notice that the parameter $t$ is uniquely determined by~\eqref{eq:exunic}; indeed this formula implies that $t$ is the ratio of translation lengths
$$\frac{\ell(\ro(g))}{\ell(g)}$$
for any hyperbolic element $g$. Therefore, Theorem~\ref{thm:exunic}, combined with Theorem~\ref{thm:KHT-char}, already imply the uniqueness result stated as Theorem~\ref{thm:uniqueness} in the Introduction --- except for two points. First, the existence result of Theorem~\ref{thm:exist} is still needed to give any substance to this uniqueness statement, namely, we must still establish the irreducibility of the representations $\ro^\infty_t$. Secondly, since Theorem~\ref{thm:uniqueness} is stated without continuity assumption, we need to establish the  automatic continuity of Theorem~\ref{thm:auto}. We defer this (independent) proof to Section~\ref{sec:auto}.

\begin{proof}[Proof of Theorem~\ref{thm:exunic}]
We choose the exhaustion by finite\hyph{}dimensional spaces $\HH^n$ in such a way that they all contain $p_0$; in particular, the stabiliser $\OO$ of $p_0$ meets each $\Isom(\HH^n)$ in a subgroup isomorphic to $\OO(n)$.

By Proposition~\ref{prop:restriction}, the restriction of $\ro$ to $\Isom(\HH^n)$ is non-elementary and hence admits a unique minimal invariant hyperbolic subspace, which we denote by $\HHI_n\se \HHI$. Note that $\HHI_n\se \HHI_{n+1}$ holds. In view of the classification that we established in Theorem~B of~\cite{Monod-Py}, there is $0<t\leq 1$ such that $\ell(\ro(g)) = t \, \ell(g)$ holds for every $g\in \Isom(\HH^n)$. In particular, it follows that $t$ does not depend on $n$.

Next we observe that $\ro(\OO)$ fixes some point $p\in\HHI$. This follows of course from the much stronger result of~\cite{Ricard-Rosendal} cited in the above proof of Proposition~\ref{prop:BO}, but it can also be deduced e.g.\ from the fact that $\OO$ has property~(T) as a polish group~\cite[Rem.~3(i)]{Bekka03}.

Let $p_n$ be the nearest-point projection of $p$ to $\HHI_n$. Then $p_n$ is fixed by $\OO(n)$ since the projection map is equivariant under $\Isom(\HH^n)$. Since the union of all $\HHI_n$ is dense in $\HHI$ by irreducibility of the representation, it follows that the sequence $(p_n)$ converges to $p$. We deduce that
$$\cosh d\big(\ro(g) p, \ro(h) p\big)  = \lim_{n\to\infty} d\big(\ro(g) p_n, \ro(h) p_n\big)$$
holds for all $g,h\in\Isomi$. We now claim that we have
$$ \lim_{n\to\infty} d\big(\ro(g) p_n, \ro(h) p_n\big) =  \big(\cosh d(g p_0, h p_0)\big)^t$$
for all $g,h$ in the union of all $\Isom(\HH^k)$. This follows from our classification of the irreducible representations of $\Isom(\HH^n)$ on $\HHI$ (Theorem~B in~\cite{Monod-Py}) together with the fact that $p_n$ is the unique point of $\HHI_n$ fixed by $\OO(n)$ (Lemma~3.9~\cite{Monod-Py}), and from the computation of the limit performed in the proof of Theorem~\ref{thm:power}.

In conclusion, we have indeed established the relation~\eqref{eq:exunic} on the union of all $\Isom(\HH^n)$, and hence on $\Isomi$ by density.
\end{proof}

\begin{rem}\label{rem:unique:O}
The above proof also shows that $p$ is the unique fixed point of $\OO$ under the (arbitrary) irreducible continuous self\hyph{}representation $\ro$.
\end{rem}

\begin{proof}[Proof of Theorem~\ref{thm:exist}]
By construction, $\ro^\infty_t$ is a continuous self\hyph{}representation such that~\eqref{eq:exunic} holds for some point $p$ with hyperbolically total orbit. We need to argue that it is irreducible.

By Corollary~\ref{cor:non-elem}, $\ro^\infty_t$ is non-elementary and therefore it contains a unique irreducible part, or equivalently preserves a unique minimal hyperbolic subspace $\HH'\se\HHI$. We now apply Theorem~\ref{thm:exunic} to the resulting representation on $\HH'$, yielding some parameter $t'$ and some point $p'\in\HH'$. We observe that $t'=t$; one way to see this is to read the translation lengths from the asymptotic formula~\eqref{eq:length}, which is unaffected by a change of base-point.

Since the orbit of $p$ under $\ro^\infty_t$ is hyperbolically total, it suffices to show that $p\in\HH'$ to deduce $\HH'=\HHI$ which completes the proof. By Remark~\ref{rem:unique:O}, the nearest-point projection of $p$ to $\HH'$ is $p'$. Choose now any $g\in \Isomi$ not in $\OO$. Then $\ro^\infty_t(g)$ does not fix $p$, and moreover the projection of $\ro^\infty_t(g)p$ to $\HH'$ is $\ro^\infty_t(g)p'$. Since~\eqref{eq:exunic} holds for both $p$ and $p'$, we have
$$d(\ro^\infty_t(g)p, p) = d(\ro^\infty_t(g)p', p').$$
The sandwich lemma (in the form of Ex.~II.2.12 in~\cite{Bridson-Haefliger}) now implies that the four points $p$, $\ro^\infty_t(g)p$, $\ro^\infty_t(g)p'$, $p'$ span a Euclidean rectangle. Since we are in a \cat{-1} space, this rectangle is degenerate and we conclude $p=p'$, as desired.
\end{proof}

\begin{rem}
We could have shortened the above proof  of Theorem~\ref{thm:exist} by \emph{defining} $\ro^\infty_t$ to be the irreducible part of the representation constructed from a kernel. However, we find that there is independent interest in knowing that a given kernel is associated to an irreducible representation.
\end{rem}

We finally turn to the M\"obius group formulation of Theorem~\ref{thm:exist}.

\begin{proof}[Proof of Theorem~\ref{thm:Mobnew}]

We use the second model for the hyperbolic space $\HHI$, as in section~\ref{sec:model}. In particular $\HHI$ sits inside the linear space $\R^2\oplus E$. 

Thinking of the representation $\ro^\infty_t$ as a self\hyph{}representation of the group $\Mob(E)$, we can assume up to conjugacy that the stabiliser of $\infty$ in $\Mob(E)$ is mapped to itself. Note that this stabiliser is precisely the group of similarities of $E$. The image under $\ro^\infty_t$ of the group of homotheties, being made of hyperbolic element of $\Mob(E)$, must then fix a unique point $q$ of $E$. Conjugating again, we assume that $q=0$, i.e. that the images of homotheties are linear maps of $E$. This also implies that $\OO(E)$ maps to itself. The assertion about homotheties in Theorem~\ref{thm:Mobnew} is now a direct consequence of the fact that $\ro^\infty_t$ multiplies translation lengths in $\HHI$ by $t$. It is also a formal consequence that the translation by $v\in E$ is sent to an isometry whose translation part has norm $c\| v\|^t$ for some constant $c>0$ independent of $v$. Conjugating once more by a homothety, we can assume that $c=1$. 

It remains only to justify that the induced self\hyph{}representation of $\Isom(E)$ is affinely irreducible: suppose that $F\se E$ is a closed affine subspace invariant under $\ro^\infty_t(\Isom(E))$. As in the proof of Theorem~\ref{thm:exunic}, we denote by $\HHI_n\se \HHI$ the unique minimal invariant hyperbolic subspace under $\Isom(\HH^n)$ seen as a subgroup of the source $\Mob(E)$. The previous discussion implies that $0\in F$, hence $F$ is linear. It further implies that $\HHI_n$ is of the form 
$$\HHI\cap \left(\R^2\oplus E_n\right)$$ 
for some increasing sequence of closed subspaces $E_n \se E$ with dense union. It thus suffices to show $E_n\se F$. This follows from Lemma~2.2 and Proposition~2.4 in~\cite{Monod-Py}.
\end{proof}

\section{Automatic continuity}\label{sec:auto}

Tsankov~\cite{Tsankov13} proved that every isometric action of $\OO$ on a Polish metric space is continuous. We do not know whether $\Isomi$ enjoys such a strong property. However, we shall be able to prove the automatic continuity of Theorem~\ref{thm:auto} by combining Tsankov's result for $\OO$ with the following fact about the local structure of $\Isomi$.

\begin{prop}\label{prop:auto}
Let $\OO$ be the stabiliser in $\Isomi$ of a point in $\HHI$ and let $g\notin\OO$. Let further $U\se \OO$ be a neighbourhood of the identity in $\OO$. Then $\OO g\inv U g \OO$ is a neighbourhood of $\OO$ in $\Isomi$.
\end{prop}

\begin{proof}
We claim that the set
$$J=\big\{ d(g p, u g p) : u\in  U \big\}$$
contains the interval $[0, \epsilon)$ for some $\epsilon>0$. Indeed, consider a (totally geodesic) copy of $\HH^2$ in $\HHI$ containing both $p$ and $g p$ and denote by $\SO(2)< \OO$ a corresponding lift of the orientation-preserving stabiliser of $p$ in $\Isom(\HH^2)^\circ$. Note that $\SO(2)$ does not fix $gp$. Since $\SO(2)$ is locally connected, there is a connected neighbourhood $V$ of the identity in $\SO(2)$ with $V\se \SO(2)\cap U$. Since moreover $\SO(2)$ is connected but does not fix $g p$, we deduce that $V$ cannot fix $g p$. Considering that $d(g p, v g p)$ is in $J$ when $v\in V$ and that $V$ is connected, the claim follows.

To establish the proposition, we shall prove that every $h\in\Isomi$ with $d(p, h p)<\epsilon$ lies in $\OO g\inv U g \OO$. By the claim, there is $u\in U$ with $d(g p, u g p)= d(p, h p)$. In other words, there is $q\in g\inv U g$ with $d(p, q p) = d(p, h p)$. By transitivity of $\OO$ on any sphere centered at $p$, there is $q'\in\OO$ with $q' q p = hp$. We conclude $h\in q' q\OO$, as claimed.
\end{proof}

\begin{proof}[Proof of Theorem~\ref{thm:auto}]
Let $\ro\colon\Isomi\to\Isomi$ be an irreducible self\hyph{}representation and let $\OO<\Isomi$  be the stabiliser of a point in $\HHI$ for its tautological representation. Then $\ro(\OO)$ has bounded orbits by~\cite[p.~190]{Ricard-Rosendal} and hence fixes a point $p\in\HHI$ by the Cartan fixed-point theorem~\cite[II.2.8]{Bridson-Haefliger}. We claim that the function
$$D\colon \Isomi \longrightarrow \R_+,  \kern10mm D(g) = d\big(\ro(g) p, p\big)$$
is continuous at $e\in\Isomi$. Let thus $(g_n)$ be a sequence converging to $e$ in $\Isomi$. We fix some $g\notin\OO$; then Proposition~\ref{prop:auto} implies that we can write $g_n = k_n g\inv u_n g k'_n$ for $k_n, u_n, k'_n\in\OO$ such that the sequence $(u_n)$ converges to~$e$. Since $\ro(\OO)$ fixes $p$, we have
$$D(g_n) = d\big(\ro(u_n)\ro(g)p, \ro(g)p\big).$$
Tsankov~\cite{Tsankov13} proved that every isometric action of $\OO$ on a Polish metric space is continuous; therefore the right hand side above converges to zero as $n\to\infty$, proving the claim.

It now follows that $D$ is continuous on all of $\Isomi$, because if $g_n\to g_\infty$ in $\Isomi$ we can estimate
$$\Big| d\big(\ro(g_n) p, p\big) - d\big(\ro(g_\infty) p, p\big) \Big| \leq d\big(\ro(g_n) p, \ro(g_\infty) p\big) = d\big(\ro(g_\infty\inv g_n) p, p\big).$$
Equivalently, the function of hyperbolic type $\cosh D$ is continuous. Since $\ro$ is irreducible, the orbit of $p$ is hyperbolically total. It follows that $\ro$ is continuous, see Theorem~\ref{thm:KHT-char} and Corollary~\ref{cor:KHT-group}.
\end{proof}

\section{Finite\hyph{}dimensional post-scripta}\label{sec:PS}
\begin{flushright}
\begin{minipage}[t]{0.7\linewidth}\itshape\small
Et je vais te prouver par mes raisonnements\ldots\\
Mais malheur \`{a} l'auteur qui veut toujours instruire\,!\\
Le secret d'ennuyer est celui de tout dire.
\begin{flushright}
\upshape\small
--- Voltaire, \itshape\ 
Sur la nature de l'homme, \upshape\ 1737\\ Vol.~I p.~953 of the 1827 edition by Jules Didot l'a\^{\i}n\'e (Paris). 

\end{flushright}
\end{minipage}
\end{flushright}

Our previous work~\cite{Monod-Py} focussed on representations $\ro^n_t$ into $\Isomi$ of the \emph{finite\hyph{}dimensional} groups $\Isom(\HH^n) \cong \PO(1,n)$. A main application was the construction of exotic locally compact deformations of the classical Lobachevsky space $\HH^n$. Nonetheless, a significant part of that article was devoted to the ana\-lysis of an equivariant map
\begin{equation*}
f^n_t\colon \HH^n \lra \HHI
\end{equation*}
canonically associated to the representation with parameter $0<t<1$. We proved notably that $f^n_t$ is a harmonic map whose image is a minimal submanifold of curvature $\frac{-n}{t(t+n-1)}$.

Furthermore, we investigated a quantity denoted by $I_u$ in~\cite[\S3.C]{Monod-Py} which, in the language of the present article, is none other than the radial function of hyperbolic type associated to $f^n_t$. More precisely, we showed that $\HHI$ contains a unique point $p$ fixed by $\OO(n)$ under $\ro^n_t$. Therefore, we have a bi-$\OO(n)$-invariant function of hyperbolic type
\begin{equation*}
F\colon  \Isom(\HH^n) \lra \R_+, \kern10mm F(g) = \cosh d\left(\ro^n_t(g) p, p\right).
\end{equation*}
Being bi-$\OO(n)$-invariant, $F$ can be represented by
\begin{equation*}
F_0\colon \R_+ \lra \R_+, \kern10mm F_0(u) = F(e^{u X}), 
\end{equation*}
where $e^{u X}$ represents any one-parameter group of hyperbolic elements for a Cartan decomposition with respect to $\OO(n)$, normalized so that $e^X$ has translation length~$1$. Now $I_u=F_0(u)$.

\medskip
It was proved in~\cite{Monod-Py} that $f^n_t$ is large-scale isometric; we now propose a more precise and more uniform statement.

\begin{prop}\label{prop:metric}
For all $u\geq 0$ we have
\begin{equation*}
\cosh(t u) \leq F_0(u) \leq \cosh^t(u).
\end{equation*}
It follows that for all $x,y\in\HH^n$ we have
\begin{equation}\label{presque-iso}
0 \leq d\left(f^n_t(x), f^n_t(y)\right) - t\, d(x,y) \leq (1-t) \log 2.
\end{equation}
\end{prop}

In other words, $f^n_t$ fails to be an isometry after rescaling by $t$ only by an additive error bounded by $\log 2$ independently of the dimension $n$ and of $t$. Of course, equation~\eqref{presque-iso} also holds for $f^\infty_t$ instead of $f^{n}_t$. This follows either by applying the same proof as below to $f^\infty_t$, or by taking the limit in equation~\eqref{presque-iso} as $n$ goes to infinity. 

For the local behaviour of $f^n_t$, we refer the reader to Proposition~3.10 in~\cite{Monod-Py}. As for the local behaviour of $f^\infty_t$, one can observe that equation~\eqref{eq:exist} from the introduction immediately implies that
\begin{equation*}
d(f^\infty_t(x),f^\infty_t(y))=\sqrt{t}\,d(x,y)+O\left(d(x,y)^2\right)
\end{equation*} 
as $d(x,y)$ goes to $0$. 

Recalling that the usual spherical metric $d_{\bS}$ on $\bS^{n-1}$ or $\bS^\infty$ can be realized as a visual metric at infinity induced by $\HH^n$ or $\HHI$, which is defined in terms of exponentials of distances in $\HH^n$ or $\HHI$ (see e.g.~\cite[p.~434]{Bridson-Haefliger}), the following is an immediate consequence.

\begin{cor}
The map $f^n_t$ induces a bi-Lipschitz embedding of the snowflake $(\bS^{n-1}, d_{\bS}^t)$ into the round sphere $(\bS^\infty, d_{\bS})$.\qed
\end{cor}

\begin{proof}[Proof of Proposition~\ref{prop:metric}]
Thanks to Proposition~\ref{prop:restriction}, the restriction of $\ro^\infty_t$ to $\Isom(\HH^n)$ has an irreducible part, which is necessarily isomorphic to $\ro^n_t$ due to the classification of~\cite{Monod-Py}. Considering that the distances decrease when projecting to the corresponding minimal invariant hyperbolic subspace (just as in the proof of Theorem~\ref{thm:exist}), we deduce $F_0(u) \leq \cosh^t(u)$.

On the other hand, $\arcosh F_0(u)$ is always bounded below by the translation length of $e^{u X}$ under $\ro^\infty_t$, which is $t u$. This yields $F_0(u) \geq \cosh(t u)$.

Now the inequalities involving distances follow by elementary calculus methods since $\cosh d\left(f^n_t(x), f^n_t(y)\right) =F_0(u)$ when $u=\cosh d(x,y)$.
\end{proof}



\bibliographystyle{amsplain}
\bibliography{biblio_mp2}

\end{document}